\documentclass[a4paper,11pt,onecolumn,twoside]{article}
\usepackage{mathrsfs}
\usepackage{mathrsfs}
\usepackage{amsfonts}
\usepackage{fancyhdr}
\usepackage{amsmath,amsfonts,amssymb,graphicx,amsthm}
\allowdisplaybreaks
\usepackage{amsmath,latexsym,amsfonts,indentfirst,amsthm,amsxtra,amssymb,bm}

\numberwithin{equation}{section}

 \newtheorem{lemma}{Lemma}[section]
 
 \newtheorem{theorem}{Theorem}[section]
 \newtheorem{corollary}[lemma]{Corollary}
 \theoremstyle{remark}
 \newtheorem{remark}{Remark}[section]

\begin{document}

\title{\bf Remarks on large-time Behavior of Solutions to the Cauchy Problem of One-dimensional Viscous Radiative and Reactive Gas}
\author{{\bf Yongkai Liao}\\[1mm]
Institute of Applied Physics and Computational Mathematics, Beijing 100088, China\\
Email address: liaoyongkai@126.com
}
\date{}

\maketitle

\begin{abstract}
This paper is concerned with the large-time behavior of solutions to the Cauchy problem of the one-dimensional viscous radiative and reactive gas. Based on the elaborate energy estimates, we developed a new approach to derive the upper bound of the absolute temperature. Our results have improved the results obtained in Liao and Zhao [{\it J. Differential Equations} {\bf 265} (2018), no.5, 2076-2120].   \vspace{3mm}

\end{abstract}


\section{Introduction and main result}
In this paper we study the large-time behavior of global smooth solutions to the Cauchy problem of a model for the one-dimensional viscous radiative and reactive gas. The model consists of the following equations in the Lagrangian coordinates corresponding to the conservation laws of the mass, the momentum and the energy coupling with a reaction-diffusion equation (cf. \cite{Liao-Zhao,Liao-Zhao-JDE-2018, Umehara-Tani-JDE-2007, Umehara-Tani-PJA-2008}):
\begin{eqnarray}\label{a1}
    v_t-u_x&=&0,\nonumber\\
    u_t+p\left(v,\theta\right)_x&=&\left(\frac{\mu u_x}{v}\right)_x,\\
    e_t+p(v,\theta)u_x&=&\frac{\mu u_{x}^{2}}{v}+\left(\frac{\kappa\left(v,\theta\right)\theta_{x}}{v}\right)_{x}+\lambda\phi z,\nonumber\\
    z_{t}&=&\left(\frac{dz_x}{v^{2}}\right)_{x}-\phi z.\nonumber
\end{eqnarray}
Here $x\in\mathbb{R}$ is the Lagrangian space variable, $t\in\mathbb{R}^+$ the time variable and the primary dependent variables are the specific volume $v=v\left(t,x\right)$, the velocity\,$u=u\left(t,x\right)$, the absolute temperature\, $\theta=\theta\left(t,x\right)$ and the mass fraction of the reactant\, $z=z\left(t,x\right)$. The positive constants $d$ and $\lambda$ are the species diffusion coefficient and the difference in the heat between the reactant and the product, respectively. The reaction rate function $\phi=\phi\left(\theta\right)$ is defined, from the Arrhenius law \cite{Ducomet-Zlotnik-NonliAnal-2005,Liao-Zhao,Liao-Zhao-JDE-2018, Umehara-Tani-JDE-2007, Umehara-Tani-PJA-2008}, by
\begin{equation}\label{a2}
 \phi\left(\theta\right)=K\theta^{\beta}\exp\left(-\frac{A}{\theta}\right),
\end{equation}
where positive constants $K$ and $A$ are the coefficients of the rates of the reactant and the activation energy, respectively, and $\beta$ is a non-negative number.

We treat the radiation as a continuous field and consider both the wave and photonic effect. Assume that the high-temperature radiation is at thermal equilibrium with the fluid. Then the pressure $p$ and the specific internal energy $e$ consist of a linear term in $\theta$ corresponding to the perfect polytropic contribution and a fourth-order radiative part due to the Stefan-Boltzmann radiative law \cite{Mihalas-Mihalas-1984,Umehara-Tani-JDE-2007}:
\begin{equation}\label{a3}
  p\left(v,\theta\right)=\frac{R\theta}{v}+\frac{a\theta^{4}}{3}, \quad e(v,\theta)=C_{v}\theta+av\theta^{4},
\end{equation}
where the positive constants $R$, $C_{v}$, and $a$ are the perfect gas constant, the specific heat and the radiation constant, respectively.

As in \cite{Liao-Zhao, Umehara-Tani-JDE-2007, Umehara-Tani-PJA-2008}, we also assume that the bulk viscosity $\mu$ is a positive constant and the thermal conductivity $\kappa=\kappa\left(v,\theta\right)$ takes the form
\begin{equation}\label{a4}
\kappa\left(v,\theta\right)=\kappa_{1}+\kappa_{2}v\theta^{b}
\end{equation}
with $\kappa_{1}$, $\kappa_{2}$ and $b$ being some positive constants. Moreover, as pointed out in \cite{Jiang-ZHeng-JMP-2012,Jiang-ZHeng-ZAMP-2014}, the most physically interesting radiation case is $b=3$. The initial data is given by
  \begin{equation}\label{a5}
 \left(v\left(0,x\right),u\left(0,x\right),\theta\left(0,x\right), z\left(0,x\right)\right)=\left(v_0\left(x\right),u_0\left(x\right),\theta_{0}\left(x\right), z_{0}\left(x\right)\right)
\end{equation}
for $x\in\mathbb{R}$, which is assumed to satisfy the following far-field condition:
\begin{equation}\label{a6}
\lim_{|x|\rightarrow\infty}\left(v_0\left(x\right),u_0\left(x\right),\theta_{0}\left(x\right), z_{0}\left(x\right)\right)=(1,0,1,0).
 \end{equation}

The global well-posedness of large solutions to initial/initial-boundary value problems of the system \eqref{a1}-\eqref{a5} has been paid a lot of attention by many authors. For the multi-dimensional case, the global existence of variational solutions have been studied by  Donatelli and Trivisa \cite{Donatelli-Trivisa-CMP-2006}, Feireisl and Novotn\'{y} \cite{Feireisl-Novotny-Book-2009}. For the construction of global spherically (cylindrically) symmetric flows to \eqref{a1}-\eqref{a5}, the global well-posedness of solutions have been studied by many authors. We can refer to \cite{Liao-Wang-Zhao-2018,Qin-Zhang-Su-Cao-JMFM-2016,Zhang-Non-2017} and the references therein.

For the 0ne-dimensional case, Umehara and Tani \cite{Umehara-Tani-JDE-2007} established global existence, uniqueness of classical solutions for $4\leq b\leq 16$ and $0\leq\beta\leq\frac{13}{2}$ with the free and pure Neumann boundary conditions. Soon after they improved their results in \cite{Umehara-Tani-PJA-2008} for the case $b\geq 3$ and $0\leq\beta\leq b+9$. Moreover, Qin et.al.
\cite{Qin-Hu-Wang-Huang-Ma-JMAA-2013} improved the results for $\left(b, \beta\right)\in E$, where $E=E_{1}\bigcup E_{2}$ with
 \begin{eqnarray*}
E_{1}&&=\left\{\left(b,\beta\right)\in\mathbb{R}^{2}:\quad \frac{9}{4}< b< 3,\ 0\leq \beta< 2b+6\right\},\\
E_{2}&&=\left\{\left(b,\beta\right)\in\mathbb{R}^{2}:\quad 3\leq b,\ 0\leq \beta<b+9\right\}.
 \end{eqnarray*}
Furthermore, Jiang and Zheng \cite{Jiang-ZHeng-JMP-2012} studied global solvability and asymptotic behavior for the above free and pure Neumann boundary value problem under the assumptions  $b\geq 2$ and $0\leq\beta\leq b+9$.

For the Dirichlet-Neumann boundary value problem, Ducomet \cite{Ducomet-MMAS-1999} established the global existence and exponential decay of solutions in $H^{1}$ for the case $b\geq 4$. Later on, Qin et.al. \cite{Qin-Hu-Wang-QAM-2011} strengthened the above results for $\left(b, \beta\right)\in \widetilde{E}$, where $\widetilde{E}=E_{3}\bigcup E_{4}$ with
 \begin{eqnarray*}
E_{3}&&=\left\{\left(b,\beta\right)\in\mathbb{R}^{2}:\quad 2< b< 3,\ 0\leq \beta< 2b+6\right\},\\
E_{4}&&=\left\{\left(b,\beta\right)\in\mathbb{R}^{2}:\quad 3\leq b,\ 0\leq \beta<b+9\right\}.
 \end{eqnarray*}
Moreover, Jiang and Zheng \cite{Jiang-ZHeng-ZAMP-2014} improved the result for the case of $b\geq 2$ and $0\leq\beta\leq b+9$. We also refer the readers to \cite{Ducomet-Feireisl-CMP-2006,Qin-Hu-JMP-2011,Qin-Liu-Yang-JDE-2012,Zhang-Xie-JDE-2008} for related studies on the one-dimensional model for thermally radiation magnetohydrodynamics.

For the Cauchy problem of \eqref{a1}-\eqref{a6}, results on the global solvability and the precise description of large-time behavior of the global solution constructed are obtained by Liao and Zhao in \cite{Liao-Zhao-JDE-2018} for $b\geq \frac{11}{3}$ and $0\leq\beta\leq b+9$. Note that the result obtained in \cite{Liao-Zhao-JDE-2018} excludes the most interesting physical case $b=3$. We can also refer to \cite{He-Liao-Wang-Zhao-2017,Liao-Xu-Zhao-SCM-2018,Liao-Zhang-JMAA-2018,Liao-Zhao} for more references therein.

The aim of this paper is to improve the result in \cite{Liao-Zhao-JDE-2018} with the larger range of $\left(b, \beta\right)$ than that in \cite{Liao-Zhao-JDE-2018}. More precisely, our main result can be stated as follows
\begin{theorem}\label{Th1.1}
Suppose that
\begin{itemize}
\item  The parameters $b$ and $\beta$ are assumed to satisfy:
\begin{equation}\label{1.9}
 b>\frac{12}{7}, \quad 0\leq\beta< b+9;
 \end{equation}
\item The initial data $ \left(v_{0}(x), u_{0}(x), \theta_{0}(x), z_{0}(x)\right)$ satisfy
\begin{eqnarray}\label{a9}
  \left(v_{0}(x)-1, u_{0}(x ), \theta_{0}(x)-1\right)\in H^{1}\left(\mathbb{R}\right),\nonumber\\
    z_{0}(x)\in L^{1}\left(\mathbb{R}\right)\cap H^{1}\left(\mathbb{R}\right),\\
     \inf\limits_{x\in\mathbb{R}}v_{0}\left(x\right)>0, \quad\inf\limits_{x\in\mathbb{R}} \theta_{0}\left(x\right)>0, \quad 0\leq z_{0}\left(x\right)\leq 1,  \quad \forall x\in\mathbb{R}.\nonumber
  \end{eqnarray}
\end{itemize}
Then the system (\ref{a1})-(\ref{a6}) admits a unique global solution $\left(v(t,x), u(t,x), \theta(t,x), z(t,x)\right)$ which satisfies
\begin{eqnarray}\label{a11}
\underline{V}\leq v(t,x)&\leq&\overline{V},\nonumber\\
\underline{\Theta}\leq\theta(t,x)&\leq&\overline{\Theta},\\
0\leq z(t,x)&\leq& 1\nonumber
\end{eqnarray}
for all $\left(t,x\right)\in [0,\infty)\times\mathbb{R}$ and
\begin{equation}\label{a12}
\sup\limits_{0\leq t<\infty}\big\|\big(v-1, u, \theta-1, z\big)(t)\big\|_{H^{1}(\mathbb{R})}^{2}+\int_{0}^{\infty}\left(\left\|v_{x}(s)\right\|^{2}_{L^2(\mathbb{R})}+\big\|\big(u_{x}, \theta_{x}, z_{x}\big)(s)\big\|^{2}_{H^{1}\left(\mathbb{R}\right)}\right)ds\leq C.
\end{equation}
Here $\underline{V},$ $\overline{V},$ $\underline{\Theta},$ $\overline{\Theta}$ and $C$ are some positive constants which depend only on the initial data $(v_{0}(x), u_{0}(x), $
$\theta_{0}(x), z_{0}(x))$.

Moreover, the large-time behavior of the global solution $\left(v(t,x), u(t,x), \theta(t,x), z(t,x)\right)$ constructed above can be described by the non-vacuum equilibrium state $(1,0,1,0)$ in the sense that
\begin{equation}\label{a13}
\lim_{t\rightarrow+\infty}\left(\left\|\left(v-1, u, \theta-1, z\right)\left(t\right)\right\|_{L^{p}\left(\mathbb{R}\right)}
+\left\|\left(v_{x}, u_{x}, \theta_{x}, z_{x}\right)\left(t\right)\right\|_{L^{2}\left(\mathbb{R}\right)}\right)=0
\end{equation}
holds for any $p\in (2,\infty]$.
\end{theorem}

\begin{remark} Some remarks concerning Theorem \ref{Th1.1} are listed below:
\begin{itemize}
\item Obviously, our main result includes the most interesting physical case and the range of $\left(b, \beta\right)$ is larger than that in \cite{Ducomet-MMAS-1999,Jiang-ZHeng-JMP-2012,Jiang-ZHeng-ZAMP-2014,Liao-Zhao-JDE-2018,
    Qin-Hu-JMP-2011,Qin-Hu-Wang-Huang-Ma-JMAA-2013,Zhang-Xie-JDE-2008};
\item It is worth pointing out Theorem 1.1 have also removed the condition $\partial_{xx}u_{0}(x)\in L^{2}\left(\mathbb{R}\right)$ on the initial data as required in \cite{Liao-Zhao-JDE-2018}. Note that such a condition was to bound the term $\left\|u_{t}(t)\right\|^2_{L^2(\mathbb{R})}$ in \cite{Liao-Zhao-JDE-2018} (see (2.95) in \cite{Liao-Zhao-JDE-2018});

\item To study the global solovability and asymptotic behavior of global solutions to the initial-boundary value problem of system  (\ref{a1})-(\ref{a5})($x\in[0,1]$), the following assumptions in the heat conductivity $\kappa=\kappa\left(v,\theta\right)$ are imposed in \cite{Jiang-ZHeng-JMP-2012,Jiang-ZHeng-ZAMP-2014}:\\
\noindent (F1) $\kappa$ is strictly positive, i.e., there exists a positive constant $\underline{\kappa}$ such that $0<\underline{\kappa}\leq \kappa\left(v,\theta\right)$.\\
\noindent (F2) $\kappa(\cdot,\cdot)\in C^{2,1}\left(\mathbb{R}^{+}\times\mathbb{R}^{+}\right)$. Moreover, $\kappa\left(v,\theta\right)$ satisfies the following growth conditions for some $b\geq2$:
\begin{eqnarray*}
&&\text{for any}\,\,\delta>1,\quad \forall v\in[\delta^{-1},\delta],\,\,\theta>0,\\
&&0<\kappa_{1}(\delta)(1+\theta)^{b}\leq\kappa\left(v,\theta\right),\\
&&\left(\kappa+|\kappa_{v}|+|\kappa_{\theta}|+|\kappa_{vv}|\right)\left(v,\theta\right)\leq\kappa_{2}(\delta)(1+\theta)^{b};
\end{eqnarray*}
\noindent (F3) For any $\delta>0$, there exists a constant $\kappa_{0}(\delta)>0$ and some $r\in[2,\min(8,b)]$ such that for $v\geq\delta$,
\begin{eqnarray*}
\frac{\kappa\left(v,\theta\right)}{v}\geq\kappa_{0}(\delta)\theta^{r}.
\end{eqnarray*}
Note that $(F3)$ is not necessary and recently Song et.al. have removed the assumption $(F3)$ in \cite{Song-Li-Zhang-Acta Math-2018}. Clearly, the assumptions $(F1)$ and $(F2)$ are more general than the form \eqref{a4} considered in our paper and such assumptions also give rise to additional difficulty due to high-order nonlinearity. More precisely, when $\kappa=\kappa\left(v,\theta\right)$ satisfies the assumptions $(F1)$ and $(F2)$, \eqref{b66} in Section 3 will be replaced by
\begin{eqnarray*}
\left|K_{v}(v,\theta)\right|+\left|K_{vv}(v,\theta)\right|\leq C(1+\theta)^{b+1},\quad \left|\left(\frac{\kappa}{v}\right)_{v}\right|\leq C(1+\theta)^{b},
\end{eqnarray*}
(see \eqref{b63} for the definition of function $K(v,\theta)$). Clearly, $\left|K_{v}(v,\theta)\right|$, $\left|K_{vv}(v,\theta)\right|$, and $\left|\left(\frac{\kappa}{v}\right)_{v}\right|$ behaves higher order growth in $\theta(t,x)$ than that in \eqref{b66}. However, we can still prove Theorem 1.1 when $\kappa=\kappa\left(v,\theta\right)$ satisfies the assumptions $(F1)$ and $(F2)$ ($b>\frac{12}{7}, 0\leq\beta< b+9$) by using the method developed in this paper.
\end{itemize}
\end{remark}

Now we outline the main ideas to deduce our main result obtained in Theorem \ref{Th1.1}. The key point to study the large-time behavior of the system (\ref{a1})-(\ref{a6}) is to deduce the uniform-in-time positive lower and upper bounds on both density and the absolute temperature. Motivated by the work of Jiang \cite{Jiang-CMP-1999}, Liao and Zhao \cite{Liao-Zhao-JDE-2018} have used the cut-off function $\varphi\left(x\right)$ (see \eqref{varphi}) to derive a explicit expression of $v\left(t,x\right)$ (see \eqref{b24}) to deduce uniform-in-time positive lower and upper bounds on the specific volume. Then the authors modify the argument developed in \cite{Jiang-ZHeng-JMP-2012,Jiang-ZHeng-ZAMP-2014,Kawohl-JDE-1985} to introduce following auxiliary functions
\begin{eqnarray}
X(t):&=&\int_{0}^{t}\int_{\mathbb{R}}\left(1+\theta^{b+3}(s,x)\right)\theta_{t}^{2}(s,x),\nonumber\\ Y(t):&=&\sup\limits_{s\in(0,t)}\int_{\mathbb{R}}\left(1+\theta^{2b}(s,x)\right)\theta_{x}^{2}(s,x)dx,\nonumber\\
Z(t):&=&\sup\limits_{s\in(0,t)}\int_{\mathbb{R}}u_{xx}^{2}(s,x)dx,\nonumber\\
W(t):&=&\int_{0}^{t}\int_{\mathbb{R}}u^{2}_{xt}(s,x)\nonumber
\end{eqnarray}
to deduce the special relationships between them:
\begin{eqnarray}\label{a14}
&&X(t)+Y(t)\leq C\left(1+Z(t)^{\lambda_{1}}\right)+\epsilon W(t),\nonumber\\
&&Z(t)\leq C\left(1+X(t)+Y(t)+Z(t)^{\lambda_{2}}\right).\quad(0<\lambda_{1}, \lambda_{2}, \epsilon<1)
\end{eqnarray}
Based on \eqref{a14} and \eqref{b58}, we can deduce the uniform-in-time upper bound on the absolute temperature. The purpose of introducing function $W(t)$ is to make the estimate more delicate in estimating the term $I_{8}$ (see \eqref{b68} in Section 3, see also (2.70) in \cite{Liao-Zhao-JDE-2018}). Besides, the role of $Z(t)$ is to bound the terms $\|u_{x}(t)\|_{L^{\infty}(\mathbb{R})}$ and $\|u_{x}(t)\|_{L^{2}(\mathbb{R})}$ in deducing \eqref{a14} (see (2.55) and (2.56) in \cite{Liao-Zhao-JDE-2018}). Note that the range of $\left(b, \beta\right)$ is determined by the fact that $0<\lambda_{1}, \lambda_{2}<1$. In order to broaden the range of $\left(b, \beta\right)$, our main idea is to get rid of using the auxiliary functions $Z(t)$ and $W(t)$. The key point in our analysis can be outlined as follows:

\begin{itemize}
\item [(i).] After deducing the uniform-in-time lower bound of $v(t,x)$, Lemma 2.3 can be improved to the case $0\leq m\leq\frac{b+4}{2}$ (see Corollary 2.5) due to \eqref{b38};

\item [(ii).] We employ Corollary 2.5 to deduce a more delicate energy estimate on $\|v_{x}(t)\|_{L^{2}(\mathbb{R})}$ than that deduced in \cite{Liao-Zhao-JDE-2018}. It is worth pointing out that one can deduce $\|v_{x}(t)\|^{2}_{L^{2}(\mathbb{R})}\leq C$ in \cite{Jiang-ZHeng-JMP-2012,Jiang-ZHeng-ZAMP-2014} due to the following key fact (see for instance, Lemma 2.8 in \cite{Jiang-ZHeng-ZAMP-2014})
\begin{eqnarray}\label{a15}
\int_{0}^{t}\left\|u(s)\right\|^{2}_{L^\infty([0,1])}ds\leq C\int_{0}^{t}\left\|u(s)\right\|^{2}_{L^\infty([0,1])}ds\leq C.
\end{eqnarray}
Clear, \eqref{a15} does not hold true since $x\in\mathbb{R}$ in our case;

\item [(iii).] Based on Lemma 2.7, we then deduce estimates on $\|u_{x}(t)\|_{L^{2}(\mathbb{R})}^{2}$ and $\int_{0}^{t}\left\|u_{xx}(s)\right\|_{L^{2}(\mathbb{R})}^2ds$ in terms of $\left\|\theta\right\|_{L^\infty([0,T]\times\mathbb{R})}$ in Lemma 2.8. We emphasize that the above lemma plays a key role in avoiding the use of functions $Z(t)$ and $W(t)$;

\item [(iv).]Thanks to Lemma 2.8, the terms $I_{8}$, $I_{9}$, and $I_{10}$ will be estimated in a different way from that deduced in \cite{Liao-Zhao-JDE-2018}. We do not require integration by parts in deriving estimates on the terms $I_{8}$ and $I_{9}$. As for the term
    $\int_{0}^{t}\int_{\mathbb{R}}u_{x}^{4}$ involved in $I_{8}$ and $I_{10}$, in view of \eqref{b3}, \eqref{b50}, Gagliardo-Nirenberg, Sobolev's, and H\"{o}lder's inequality, we have
    \begin{eqnarray*}
    \int_{0}^{t}\int_{\mathbb{R}}u_{x}^{4}&\leq&
    C\int_{0}^{t}\left\|u_{x}\right\|^{2}_{L^\infty(\mathbb{R})}
     \left\|u_{x}\right\|^{2}_{L^2(\mathbb{R})}ds
     \leq C\int_{0}^{t}\left\|u_{x}\right\|^{3}_{L^2(\mathbb{R})}
    \left\|u_{xx}\right\|_{L^2(\mathbb{R})}ds\nonumber\\
    &\leq&C\int_{0}^{t}\left\|u_{x}\right\|_{L^2(\mathbb{R})}
    \left\|u_{xx}\right\|^{2}_{L^2(\mathbb{R})}ds\leq
    C\left(1+\left\|\theta\right\|^{\frac{3\ell}{2}}_{L^\infty([0,T]\times\mathbb{R})}\right),
\end{eqnarray*}
and the definition of $\ell$ is given in \eqref{b51}.

\end{itemize}

The rest of the paper is organized as follows: we first derive some useful a priori estimates and pointwise bounds on the specific volume in Section 2. Then pointwise bounds on the absolute temperature will be derived in Sections 3.
\bigbreak
\noindent{\bf Notations:}\quad Throughout this paper, $C\geq 1$ is used to denote a generic positive constant which may dependent only on $\inf\limits_{x\in\mathbb{R}}v_{0}\left(x\right)$, $\inf\limits_{x\in\mathbb{R}}\theta_{0}\left(x\right)$, $\|\left(v_{0}-1, u_{0}, \theta_{0}-1, z_{0}\right)\|_{H^{1}{(\mathbb{R}})}$ and $\|z_{0}\|_{L^{1}{(\mathbb{R}})}$, $C\left(\cdot,\cdot\right)$ stands for some generic positive constant depending only on the quantities listed in the parenthesis and $\epsilon$ represents some small positive constant. Note that these constants may vary from line to line. The symbol $A\lesssim B$ (or $B\gtrsim A$) means that $A\leq C B$ holds uniformly for some uniform-in-time constant $C$, while $A \sim A'$ means that $A\lesssim A'$ and $A'\lesssim A$. For function spaces, ~$L^q\left(\mathbb{R}\right)\left(1\leq q\leq \infty\right)$~denotes the usual Lebesgue space on~$\mathbb{R}$~with norm~$\|{\cdot}\|_{L^q}\equiv\|{\cdot}\|_{L^q\left(\mathbb{R}\right)},$ while for $l\in\mathbb{N}$, $H^l\left(\mathbb{R}\right)$~denotes the usual $l-$th order Sobolev space with norm~$\|{\cdot}\|_l\equiv\|{\cdot}\|_{H^l\left(\mathbb{R}\right)}$. For simplicity, we use $\|{\cdot}\|_{\infty}$ to denote the norm in $L^{\infty}\left([0,T]\times\mathbb{R}\right)$ for some $T>0$ and use $\|{\cdot}\|$ to denote the norm ~$\|{\cdot}\|_{L^2\left(\mathbb{R}\right)}$.

 \section{Some a priori estimates and pointwise bounds on $v(t,x)$ }
In this section, we will deduce certain a priori estimates on the solutions of the Cauchy problem \eqref{a1}-\eqref{a6} in terms of the initial data $(v_0(x), u_0(x), \theta_0(x), z_0(x))$. To this end, for some constants $0<T\leq +\infty$, $0<M_1<M_2,$ $0<N_1<N_2$, we first define the set of functions $X(0,T;M_1,M_2;N_1,N_2)$ for which we seek the solution of the Cauchy problem \eqref{a1}-\eqref{a6} as follows:
\begin{eqnarray*}
   &&X(0, T;M_1,M_2;N_1,N_2)\\
&:=&\left\{(v(t,x), u(t,x),\theta(t,x),z(t,x))\ \left|
   \begin{array}{c}
   0\leq z(t,x)\in  C\left(0,T;H^{1}\left(\mathbb{R}\right)\cap L^1(\mathbb{R})\right),\\
   \left(v\left(t,x\right)-1,u\left(t,x\right),\theta\left(t,x\right)-1\right)\in C\left(0,T;H^{1}\left(\mathbb{R}\right)\right),\\
   \left(u_{x}\left(t,x\right), \theta_{x}\left(t,x\right), z_{x}\left(t,x\right)\right)\in L^{2}\left(0,T;H^{1}\left(\mathbb{R}\right)\right),\\
   v_{x}\left(t,x\right)\in L^{2}\left(0,T; L^{2}\left(\mathbb{R}\right)\right),\\
   M_1\leq v(t,x)\leq M_2,\ \forall (t,x)\in[0,T]\times\mathbb{R},\\
   N_1\leq \theta(t,x)\leq N_2,\ \forall (t,x)\in[0,T]\times\mathbb{R}
   \end{array}
   \right.
   \right\}.
  \end{eqnarray*}

The standard local wellposedness result on the Cauchy problem of the hyperbolic-parabolic coupled system tells us that there exists a sufficiently small positive constant $t_1>0$, which depends only on $m_0=\inf\limits_{x\in\mathbb{R}} v_0(x), n_0=\inf\limits_{x\in\mathbb{R}} \theta_0(x)$ and $\ell_0=\|(v_0-1,u_0,\theta_0-1,z_0)\|_1$ such that the Cauchy problem \eqref{a1}-\eqref{a6} admits a unique solution $(v(t,x), u(t,x),\theta(t,x),z(t,x))\in X(0,t_1;m_0/2,2+2\ell_0;n_0/2,2+2\ell_0)$. Now suppose that such a solution has been extended to the time step $t=T\geq t_1$ and $(v(t,x), u(t,x),\theta(t,x),z(t,x))\in X(0,T;M_1,M_2;N_1,N_2)$ for some positive constants $M_2>M_1>0, N_2>N_1>0$, we now try to deduce certain a priori energy type estimates on $(v(t,x), u(t,x),\theta(t,x),z(t,x))$ in terms of the initial data $(v_0(x), u_0(x), \theta_0(x), z_0(x))$.

The first lemma is concerned with the basic energy estimates, which will play a fundamental role in our argument.

\begin{lemma} [Basic energy estimates] Suppose that $(v(t,x), u(t,x),\theta(t,x),z(t,x))\in X(0,T;M_1,M_2;$ $N_1,N_2)$ for some positive constants $T>0, M_2>M_1>0, N_2>N_1>0$, then for all $0\leq t\leq T$, we have
\begin{equation}\label{b1}
   \int_{\mathbb{R}}z(t,x)dx+\int_{0}^{t}\int_{\mathbb{R}}\phi(s,x) z(s,x)\lesssim 1,
\end{equation}
\begin{equation}\label{b2}
 \int_{\mathbb{R}}z^{2}(t,x)dx+\int_{0}^{t}\int_{\mathbb{R}}\left(\frac{d}{v^{2}}z_{x}^{2}+\phi z^{2}\right)(s,x)
 \lesssim 1,
\end{equation}
\begin{eqnarray}\label{b3}
&&\int_{\mathbb{R}}\left(\eta(t,x)+u^2(t,x)\right)dx
+\int_{0}^{t}\int_{\mathbb{R}}\left(\frac{\mu u_{x}^{2}}{v\theta}+\frac{\kappa(v,\theta)\theta_{x}^{2}}{v\theta^{2}}+\frac{\lambda\phi z}{\theta}\right)(s,x)\lesssim 1,\\
&&\text{here}\quad \eta(v,\theta)=C_{v}\left(\theta-\ln\theta-1\right)+R\left(v-\ln v-1\right)+\frac{1}{3}av\left(\theta-1\right)^{2}\left(3\theta^{2}+2\theta+1\right),\nonumber
\end{eqnarray}
\begin{equation}\label{b16}
  0\leq z\left(t,x\right)\leq 1.
  \end{equation}
\end{lemma}
The proof of Lemma 2.1 is the same as that in \cite{Liao-Zhao-JDE-2018} and \cite{Chen-SIMA-1992}, we omit the proof for brevity.

The next lemma follows from utilizing \eqref{b3} and Jensen's inequality.

\begin{lemma}  Assume that the conditions listed in Theorem \ref{Th1.1} hold.
Then for all $k\in\mathbb{N}$ and $t\in[0,T]$, there exist $a_{k}(t), b_{k}(t)\in \Omega_{k}:=[-k-1,k+1]$ such that
 \begin{align}\label{es3.1}
 \int_{\Omega_{k}}v(t,x)\mathrm{d}x\sim 1,\quad \int_{\Omega_{k}}\theta(t,x)\mathrm{d}x\sim 1,\quad
v(t,a_{k}(t))\sim 1,\quad \theta(t,b_{k}(t))\sim 1.
  \end{align}
\end{lemma}

The following lemma gives a rough estimate on $\theta(t,x)$ in terms of the entropy dissipation rate functional
$V(t)=\displaystyle\int_{\mathbb{R}}\left(\frac{\mu u_{x}^{2}}{v\theta}+\frac{\kappa(v,\theta)\theta_{x}^{2}}{v\theta^{2}}\right)(t,x)dx$, which is a consequence of \eqref{a4} and \eqref{es3.1}.
\begin{lemma}For $0\leq m\leq\frac{b+1}{2}$ and each $x\in\mathbb{R}$ (without loss of generality, we can assume that $x\in\Omega_k$ for some $k\in\mathbb{Z}$), we can deduce that
\begin{equation}\label{b19}
\left |\theta^{m}\left(t,x\right)-\theta^{m}\left(t, b_{k}\left(t\right)\right)\right|\lesssim V^{\frac{1}{2}}\left(t\right)
\end{equation}
holds for $0\leq t\leq T$ and consequently
\begin{equation}\label{b20}
 \left|\theta\left(t,x\right)\right|^{2m}\lesssim 1+V\left(t\right), \quad x\in\overline{\Omega}_{k},\ \  0\leq t\leq T.
\end{equation}
\end{lemma}

By defining the cut-off function $\varphi\in W^{1,\infty}(\mathbb{R})$ as follows:
\begin{align}\label{varphi}
 \varphi(x)=
 \begin{cases}
 1, & \quad 0\leq x\leq k+1,\\
 k+2-x, & \quad k+1\leq x\leq k+2,\\
 0, & \quad x\geq k+2,
 \end{cases}
\end{align}
we can deduce a local representation of $v(t,x)$ in the following lemma.

\begin{lemma} Under the assumptions stated in Lemma 2.1, we have for each $0\leq t\leq T$ that
\begin{eqnarray}\label{b24}
v\left(t,x\right)=B\left(t,x\right)Q\left(t\right)+\frac{1}{\mu}\int_{0}^{t}\frac{B\left(t,x\right)Q\left(t\right)v\left(s,x\right)p\left(s,x\right)}{B\left(s,x\right)Q\left(s\right)}ds,
\quad x\in\overline{\Omega}_k,
\end{eqnarray}
here
\begin{eqnarray}\label{b25}
B\left(t,x\right)&&:=v_{0}\left(x\right)\exp\left\{\frac{1}{\mu}\int_{x}^{\infty}\left(u_{0}\left(y\right)-u\left(t,y\right)\right)\varphi\left(y\right)dy\right\},\nonumber\\
Q\left(t\right)&&:=\exp\left\{\frac{1}{\mu}\int_{0}^{t}\int_{k+1}^{k+2}\left(\frac{\mu u_{x}}{v}-p(v,\theta)\right)\left(s,y\right)\right\}.
\end{eqnarray}
\end{lemma}

Now we are turning to deduce pointwise bounds on the specific volume. Since the strategy we adopted is almost the same as that in \cite{Liao-Zhao-JDE-2018}, we only point out the details which will influence the range of $b$. By repeating the argument used in \cite{Liao-Zhao-JDE-2018}, we have
\begin{eqnarray}\label{b29}
B\left(t,x\right)\sim 1, \quad\frac{Q\left(t\right)}{Q\left(s\right)}\lesssim\exp\left\{-C(t-s)\right\},\quad\forall x\in\overline{\Omega}_{k},
\end{eqnarray}
and
\begin{eqnarray}\label{b38}
v\left(t,x\right)\geq C(k),\quad\forall x\in\overline{\Omega}_{k}.
\end{eqnarray}
It is worth to point out that we have used Lemma 2.3 ($m=\frac{1}{2}$, $b\geq 0$) to obtain the following inequality
\begin{eqnarray*}
\theta\left(t,x\right)\gtrsim 1-CV(t)
\end{eqnarray*}
to deduce \eqref{b38}.

Before continuing our discussion, we try to improve Lemma 2.3 to the case $0\leq m\leq\frac{b+4}{2}$. In fact, we have the following Corollary
\begin{corollary}
For each $x\in\mathbb{R}$, \eqref{b19} and \eqref{b20} still hold true when $0\leq m\leq\frac{b+4}{2}$.
\end{corollary}
\begin{proof} For $x\in\Omega_k$, we have from \eqref{a4} that
\begin{eqnarray}\label{b21}
 \left|\theta^{m}\left(t,x\right)-\theta^{m}\left(t, b_{k}\left(t\right)\right)\right|&&\lesssim\int_{\Omega_{k}}\big|\theta^{m-1}\theta_{x}\big|dx
 \lesssim\left(\int_{\Omega_{k}}\frac{v\theta^{2m}}{1+v\theta^{b}}dx\right)^{\frac{1}{2}}V^{\frac{1}{2}}(t).
\end{eqnarray}
By virtue of \eqref{b3}, \eqref{b38}, the boundedness of $\Omega_k$, and the fact that $0\leq m\leq\frac{b+4}{2}$, one has
\begin{equation}\label{b22}
\int_{\Omega_{k}}\frac{v\theta^{2m}}{1+v\theta^{b}}dx\lesssim\int_{\Omega_{k}}\left(v+1+\left(\theta-1\right)^{4}\right)dx\lesssim 1.
\end{equation}

Combining (\ref{b21}) and (\ref{b22}), we can complete the proof of this corollary.
\end{proof}

Now we are turning to deduce the pointwise upper bound on $v(t,x)$. We can deduce from \eqref{b24} and \eqref{b29} that
\begin{eqnarray}\label{b39}
v\left(t,x\right)&&\lesssim Q\left(t\right)+\int_{0}^{t}\frac{Q\left(t\right)}{Q\left(s\right)}v\left(s, x\right)p\left(s, x\right)ds\lesssim 1+\int_{0}^{t}\exp\left\{-C(t-s)\right\}\left(\theta+v\theta^{4}\right)\left(s, x\right)ds.
 \end{eqnarray}
Setting $m=\frac{1}{2}$ and $m=2$ in Corollary 2.5, respectively, we can get
\begin{eqnarray}\label{b40}
\int_{0}^{t}\exp\left\{-C(t-s)\right\}\theta\left(s, x\right)ds&&\lesssim\int_{0}^{t}\exp\left\{-C(t-s)\right\}(1+V\left(s\right))ds\lesssim 1,
\end{eqnarray}
and
\begin{eqnarray}\label{b41}
\int_{0}^{t}\exp\left\{-C(t-s)\right\}v(s,x)\theta^{4}\left(s, x\right)ds\lesssim\int_{0}^{t}v\left(s,x\right)\exp\left\{-C(t-s)\right\}(1+V\left(s\right))ds.
\end{eqnarray}
Then plugging \eqref{b40} and \eqref{b41} into \eqref{b39} and by using Gronwall's inequality, we can deduce that
\begin{eqnarray}\label{b42}
v\left(t,x\right)\leq C\left(k\right).
\end{eqnarray}
Combing \eqref{b38}, \eqref{b42}, and the fact that $v\left(t,x\right)-1\in C\left(0,T;H^{1}\left(\mathbb{R}\right)\right)$, we can deduce the following lemma
\begin{lemma} Under the assumptions listed in Lemma 2.1, there exist positive constants $\underline{V}$, $\overline{V}$, which depend only on the initial data $(v_0(x), u_0(x), \theta_0(x), z_0(x))$, such that
\begin{equation}\label{b43}
 \underline{V}\leq v\left(t,x\right)\leq \overline{V}
\end{equation}
holds for all $(t,x)\in[0,T]\times\mathbb{R}$.
\end{lemma}

Notice that we need $b\geq 3$ to deduce pointwise bounds on $v\left(t,x\right)$ in \cite{Liao-Zhao-JDE-2018}. Here we only require $b\geq 0$ to deduce Lemma 2.6 due to Corollary 2.5.

The next lemma is concerned with the estimate on $\|v_{x}(t)\|^{2}$, which will be frequently used later on.

\begin{lemma} Under the assumptions listed in Lemma 2.1, we have for any $0\leq t\leq T$ that
\begin{equation}\label{bz57}
\left\|v_{x}(t)\right\|^2+\int_{0}^{t}\left\|\sqrt{\theta(s)} v_{x}(s)\right\|^2 ds\lesssim 1+\|\theta\|_{\infty}.
\end{equation}
\end{lemma}

\begin{proof}
 Multiplying $\eqref{a1}_{2}$ by $\frac{v_{x}}{v}$ and integrating the resulting identity with respect to $t$ and $x$ over $\left(0,t\right)\times\mathbb{R}$, we arrive at
\begin{eqnarray}\label{b45}
&&\frac{\mu}{2}\int_{\mathbb{R}}\frac{v_{x}^{2}}{v^{2}}dx+R\int_{0}^{t}\int_{\mathbb{R}}\frac{\theta v_{x}^{2}}{v^{3}}\nonumber\\
&=&\int_\mathbb{R}\left(\frac{\mu}{2}\frac{v_{0x}^2}{v_0^2}-\frac{u_0v_{0x}}{v_0}\right)dx+\underbrace{\int_\mathbb{R}\frac{uv_x}{v}dx}_{I_1}
+\underbrace{\int^t_0\int_\mathbb{R}\frac{u_x^2}{v}}_{I_2}\\
&&+\underbrace{\int^t_0\int_\mathbb{R}\frac{Rv_x\theta_x}{v^2}}_{I_3}
+\underbrace{\int^t_0\int_\mathbb{R}\frac{4a\theta^3v_x\theta_x}{3v}}_{I_4}.\nonumber
\end{eqnarray}
In view of \eqref{b3}, \eqref{b43}, and Cauchy's inequality, we can get
\begin{eqnarray}\label{b46}
I_{1}+I_{2}+I_{3}&&\leq\epsilon\left(\int_{\mathbb{R}}\frac{v_{x}^{2}}{v^{2}}dx+\int_{0}^{t}\int_{\mathbb{R}}\frac{\theta v_{x}^{2}}{v^{3}}\right)
+C\left(\epsilon\right)\left(\int_{\mathbb{R}}u^{2}dx+\|\theta\|_{\infty}
+\int_{0}^{t}\int_{\mathbb{R}}\frac{\kappa(v,\theta)\theta^{2}_{x}}{v\theta^{2}}\cdot\frac{\theta}{\kappa(v,\theta)}\right)\nonumber\\
&&\leq\epsilon\left(\int_{\mathbb{R}}\frac{v_{x}^{2}}{v^{2}}dx+\int_{0}^{t}\int_{\mathbb{R}}\frac{\theta v_{x}^{2}}{v^{3}}\right)
+C\left(\epsilon\right)\left(1+\|\theta\|_{\infty}\right).
\end{eqnarray}
For the term $I_{4}$, it follows from \eqref{b3}, \eqref{b43}, Corollary 2.5, and the fact $b\geq\frac{3}{2}$ that
\begin{eqnarray}\label{b47}
I_{4}&&\leq\epsilon\int_{0}^{t}\int_{\mathbb{R}}\left(1+\theta^{(7-b)_{+}}\right)v_{x}^{2}
+C\left(\epsilon\right)\int_{0}^{t}\int_{\mathbb{R}}\frac{\kappa(v,\theta)\theta^{2}_{x}}{v\theta^{2}}
\cdot\left(1+\theta\right)\nonumber\\
&&\leq\frac{\epsilon}{C}\int_{0}^{t}\int_{\mathbb{R}}\left(1+V(s)\right)v_{x}^{2}
+C\left(\epsilon\right)\left(1+\|\theta\|_{\infty}\right),
\end{eqnarray}
here $(7-b)_{+}=\max\{0,7-b\}$.
On the other hand, setting $m=1$ in Corollary 2.5, one has
\begin{eqnarray}\label{b48}
\int_{0}^{t}\int_{\mathbb{R}}v_{x}^{2}&&\leq C\int_{0}^{t}\int_{\mathbb{R}}\theta v_{x}^{2}
+C\int_{0}^{t}V^{\frac{1}{2}}(s)\int_{\mathbb{R}}v_{x}^{2}dxds\nonumber\\
&&\leq\frac{1}{2}\int_{0}^{t}\int_{\mathbb{R}}v_{x}^{2}+C\int_{0}^{t}\int_{\mathbb{R}}\theta v_{x}^{2}+C\int_{0}^{t}V(s)\int_{\mathbb{R}}v_{x}^{2}dxds
\end{eqnarray}
Thus \eqref{b47} and \eqref{b48} tell us that
\begin{eqnarray}\label{b49}
I_{4}\leq\epsilon\int_{0}^{t}\int_{\mathbb{R}}\left(\theta+V(s)\right)v_{x}^{2}
+C\left(\epsilon\right)\left(1+\|\theta\|_{\infty}\right).
\end{eqnarray}
Plugging \eqref{b46}-\eqref{b49} into \eqref{b45}, choosing $\epsilon>0$ small enough, and making use of Gronwall's inequality, we can obtain \eqref{bz57}.

\end{proof}

Based on Lemma 2.7, we can deduce the following key lemma concerning on $\|u_{x}(t)\|^{2}$ and $\int_{0}^{t}\left\|u_{xx}(s)\right\|^2ds$. In fact, we have
\begin{lemma} Under the assumptions listed in Lemma 2.1, we have for any $0\leq t\leq T$ that
\begin{eqnarray}\label{b50}
\|u_{x}(t)\|^{2}+\int_{0}^{t}\left\|u_{xx}(s)\right\|^2ds\lesssim 1+\|\theta\|^{\ell}_{\infty},
\end{eqnarray}
where
\begin{eqnarray}\label{b51}
\ell=\max\{3,(8-b)_{+}\}.
\end{eqnarray}
\end{lemma}
\begin{proof} Multiplying $\eqref{a1}_{2}$ by $u_{xx}$ yields
\begin{eqnarray}
\left(\frac{u_{x}^{2}}{2}\right)_t+\frac{\mu u^{2}_{xx}}{v}
=\left(u_{x}u_{t}\right)_{x}+\frac{\mu u_{x}v_{x}u_{xx}}{v^{2}}+\frac{R\theta_{x}u_{xx}}{v}-\frac{R\theta v_{x}u_{xx}}{v^{2}}
+\frac{4a\theta^{3}\theta_{x}u_{xx}}{3}.
\end{eqnarray}
Integrating this last identity over $[0,t ]\times\mathbb{R}$ and then using \eqref{b3}, \eqref{b43}, \eqref{bz57}, Cauchy's and Sobolev's inequality, one has
\begin{eqnarray}
&&\|u_{x}(t)\|^{2}+\int_{0}^{t}\left\|u_{xx}(s)\right\|^2ds\nonumber\\
&\leq&\epsilon\int_{0}^{t}\left\|u_{xx}(s)\right\|^2ds
+C(\epsilon)\left(\int_{0}^{t}\int_{\mathbb{R}}\left(\frac{\kappa(v,\theta)\theta^{2}_{x}}{v\theta^{2}}
\cdot\frac{\theta^{2}+\theta^{8}}{1+\theta^{b}}+\theta^{2}v_{x}^{2}\right)\right)
+\int_{0}^{t}\left\|u_{x}\right\|^{2}_{L^\infty}\left\|v_{x}\right\|^{2}d\tau\nonumber\\
&\leq&\epsilon\int_{0}^{t}\left\|u_{xx}(s)\right\|^2ds
+C(\epsilon)\left(1+\|\theta\|^{\max\{2,(8-b)_{+}\}}_{\infty}
+\int_{0}^{t}\left\|u_{x}\right\|\left\|u_{xx}\right\|\left\|v_{x}\right\|^{2}d\tau\right)\nonumber\\
&\leq&2\epsilon\int_{0}^{t}\left\|u_{xx}(s)\right\|^2ds
+C(\epsilon)\left(1+\|\theta\|^{\max\{2,(8-b)_{+}\}}_{\infty}
+\int_{0}^{t}\left\|u_{x}\right\|^{2}\left\|v_{x}\right\|^{4}d\tau\right)\nonumber\\
&\leq&2\epsilon\int_{0}^{t}\left\|u_{xx}(s)\right\|^2ds
+C(\epsilon)\left(1+\|\theta\|^{\max\{3,(8-b)_{+}\}}_{\infty}\right).
\end{eqnarray}
Choosing $\epsilon>0$ small enough, we can complete the proof of our lemma.

\end{proof}

\section{Pointwise bounds on $\theta(t,x)$}
This section is devoted to deriving pointwise bounds on the absolute temperature. We first deduce the uniform-in-time upper bound on $\theta(t,x)$. To this end, we set

\begin{eqnarray}\label{b56}
X(t):&=&\int_{0}^{t}\int_{\mathbb{R}}\left(1+\theta^{b+3}(s,x)\right)\theta_{t}^{2}(s,x),\nonumber\\ Y(t):&=&\sup\limits_{s\in(0,t)}\int_{\mathbb{R}}\left(1+\theta^{2b}(s,x)\right)\theta_{x}^{2}(s,x)dx.
\end{eqnarray}

The next lemma follows directly from \eqref{b3}, \eqref{es3.1}, \eqref{b43},  and the definition of $Y(t)$.
\begin{lemma}\label{lem3.1} Under the assumptions listed in Theorem \ref{Th1.1}, we have for all $0\leq t\leq T$ that
\begin{eqnarray}\label{b58}
\|\theta(t)\|_{L^{\infty}}&\lesssim 1 +Y(t)^{\frac{1}{2b+6}}.
\end{eqnarray}
\end{lemma}

Our next result is to show that $X(t)$ and $Y(t)$ are bounded, which is different from Lemma 2.8 obtained in \cite{Liao-Zhao-JDE-2018}.
\begin{lemma} Under the assumptions listed in Lemma 2.1, we have for $0\leq t\leq T$ that
\begin{equation}\label{b62}
X(t)+Y(t)\lesssim 1.
\end{equation}
\end{lemma}
\begin{proof}
In the same manner as in \cite{Kawohl-JDE-1985,Liao-Zhao-JDE-2018,Umehara-Tani-JDE-2007}, if we set
\begin{equation}\label{b63}
K\left(v,\theta\right)=\int_{0}^{\theta}\frac{\kappa\left(v,\xi\right)}{v}d\xi=\frac{\kappa_1\theta}{v}+\frac{\kappa_2\theta^{b+1}}{b+1},
  \end{equation}
then it is easy to verify from the estimate \eqref{b43} obtained in Lemma 2.6 that
\begin{eqnarray}
K_{t}(v,\theta)&=&K_{v}(v,\theta)u_{x}+\frac{\kappa(v,\theta)\theta_{t}}{v},\label{b64}\\
K_{xt}(v,\theta)&=&\left[\frac{\kappa(v,\theta)\theta_{x}}{v}\right]_{t}+K_{v}(v,\theta)u_{xx}+K_{vv}(v,\theta)v_{x}u_{x}
+\left(\frac{\kappa(v,\theta)}{v}\right)_{v}v_{x}\theta_{t},\label{b65}\\
\left|K_{v}(v,\theta)\right|&+&\left|K_{vv}(v,\theta)\right|\lesssim\theta,\quad \left|\left(\frac{\kappa(v,\theta)}{v}\right)_{v}\right|\lesssim 1. \label{b66}
\end{eqnarray}

Multiplying $(\ref{a1})_{3}$ by $K_{t}$ and integrating the resulting identity over $\left(0,t\right)\times\mathbb{R}$, we arrive at
\begin{eqnarray}\label{b67}
&&\int_{0}^{t}\int_{\mathbb{R}}\left(e_{\theta}(v,\theta)\theta_{t}+\theta p_{\theta}(v,\theta)u_{x}-\frac{\mu u_{x}^{2}}{v}\right)K_{t}(v,\theta)\nonumber\\ &&+\int_{0}^{t}\int_{\mathbb{R}}\frac{\kappa\left(v,\theta\right)}{v}\theta_{x}K_{tx}(v,\theta)\\
&=&\int_{0}^{t}\int_{\mathbb{R}}\lambda\phi zK_{t}(v,\theta).\nonumber
\end{eqnarray}

Combining (\ref{b63})-(\ref{b67}), we have
\begin{eqnarray}\label{b68}
&&\int_{0}^{t}\int_{\mathbb{R}}\frac{e_{\theta}(v,\theta)\kappa\left(v,\theta\right)\theta_{t}^{2}}{v}
+\int_{0}^{t}\int_{\mathbb{R}}\frac{\kappa\left(v,\theta\right)\theta_{x}}{v}\left(\frac{\kappa\left(v,\theta\right)\theta_{x}}{v}\right)_{t}\nonumber\\
&\leq& C+\underbrace{\left|\int_{0}^{t}\int_{\mathbb{R}}e_{\theta}(v,\theta)\theta_{t}K_{v}(v,\theta)u_{x}\right|}_{I_5}
+\underbrace{\left|\int_{0}^{t}\int_{\mathbb{R}}\theta p_{\theta}(v,\theta)u_{x}K_{v}(v,\theta)u_{x}\right|}_{I_6}\nonumber\\
&&+\underbrace{\left|\int_{0}^{t}\int_{\mathbb{R}}\frac{\theta p_{\theta}(v,\theta)\kappa\left(v,\theta\right)u_{x}\theta_{t}}{v}\right|}_{I_7}
+\underbrace{\left|\int_{0}^{t}\int_{\mathbb{R}}\frac{\mu u_{x}^{2}K_{t}(v,\theta)}{v}\right|}_{I_8}\\
&&+\underbrace{\left|\int_{0}^{t}\int_{\mathbb{R}}\frac{\kappa(v,\theta)}{v}\theta_{x}\left(K_{vv}(v,\theta)v_{x}u_{x} +K_{v}(v,\theta)u_{xx}\right)\right|}_{I_9}\nonumber\\
&&+\underbrace{\left|\int_{0}^{t}\int_{\mathbb{R}}\frac{\kappa(v,\theta)\theta_{x}}{v} \left(\frac{\kappa(v,\theta)}{v}\right)_{v}v_{x}\theta_{t}\right|}_{I_{10}}\nonumber\\
&&+\underbrace{\left|\int_{0}^{t}\int_{\mathbb{R}}\lambda\phi zK_{v}(v,\theta)u_{x}\right|}_{I_{11}}
+\underbrace{\left|\int_{0}^{t}\int_{\mathbb{R}}\frac{\lambda\phi z\kappa\left(v,\theta\right)\theta_{t}}{v}\right| }_{I_{12}}.\nonumber
\end{eqnarray}

Now turn to control $I_k (k=5,6,\cdots, 12)$ term by term. To do so, we first have from \eqref{a3}, \eqref{a4} and \eqref{b43} that
\begin{eqnarray}\label{b69}
\int_{0}^{t}\int_{\mathbb{R}}\frac{e_{\theta}(v,\theta)\kappa\left(v,\theta\right)\theta_{t}^{2}}{v}&&\gtrsim \int_{0}^{t}\int_{\mathbb{R}}\left(1+\theta^{3}\right)\left(1+\theta^{b}\right)\theta_{t}^{2}\gtrsim X(t)
\end{eqnarray}
and
\begin{eqnarray}\label{b70}
&&\int_{0}^{t}\int_{\mathbb{R}}\frac{\kappa\left(v,\theta\right)\theta_{x}}{v}
\left(\frac{\kappa\left(v,\theta\right)\theta_{x}}{v}\right)_{t}\nonumber\\
&=&\frac{1}{2}\int_{\mathbb{R}}
\left(\frac{\kappa\left(v,\theta\right)\theta_{x}}{v}\right)^{2}(t,x)dx
-\frac{1}{2}\int_{\mathbb{R}}\left(\frac{\kappa\left(v,\theta\right)\theta_{x}}{v}\right)^{2}\left(0,x\right)dx\\
&\gtrsim& Y(t)-1.\nonumber
\end{eqnarray}

On the other hand, it follows from \eqref{b3}, \eqref{b58} and \eqref{b66} that
\begin{eqnarray}\label{b105}
|I_{5}|+|I_{6}|+|I_{7}|
&&\leq \epsilon X(t) +C\left(\epsilon\right)\left(1+\|\theta\|_{\infty}^{(6-b)_{+}}+\|\theta\|_{\infty}^{6}
+\|\theta\|_{\infty}^{b+6}\right)\int_{0}^{t}\int_{\mathbb{R}}\frac{u_{x}^{2}}{\theta}\nonumber\\
&&\leq \epsilon X(t)+C\left(\epsilon\right)\left(1+Y(t)^{\frac{b+6}{2b+6}}\right)
\leq \epsilon (X(t)+Y(t))+C\left(\epsilon\right).
\end{eqnarray}

The terms $I_{8}$, $I_{9}$, and $I_{10}$ will be bounded in a different way from that in \cite{Liao-Zhao-JDE-2018}, and we emphasize that Lemma 2.8 plays a key role in our analysis. For the term $I_{8}$, instead of integration by parts, we have from \eqref{b64} that

\begin{eqnarray}\label{b108}
I_{8}&\leq&\underbrace{\left|\int_{0}^{t}\int_{\mathbb{R}}\frac{\mu u_{x}^{3}K_{v}}{v}\right|}_{I^1_8}+\underbrace{\left|\int_{0}^{t}\int_{\mathbb{R}}\frac{\mu u_{x}^{2}\kappa(v,\theta)\theta_{t}}{v^{2}}\right|}_{I_8^2}.
\end{eqnarray}

Furthermore, in view of \eqref{b3}, \eqref{b50}, Sobolev's and Gagliardo-Nirenberg inequality, one has
\begin{eqnarray}\label{b109}
    \int_{0}^{t}\int_{\mathbb{R}}u_{x}^{4}&\lesssim&
    \int_{0}^{t}\left\|u_{x}\right\|^{2}_{L^\infty}
     \left\|u_{x}\right\|^{2}ds
     \lesssim\int_{0}^{t}\left\|u_{x}\right\|^{3}
    \left\|u_{xx}\right\|ds\nonumber\\
    &\lesssim&\int_{0}^{t}\left\|u_{x}\right\|
    \left\|u_{xx}\right\|^{2}ds\lesssim
   1+\left\|\theta\right\|^{\frac{3\ell}{2}}_{\infty}.
\end{eqnarray}
Thus one can infer from \eqref{b3}, \eqref{b43}, \eqref{b66}, \eqref{b109}, the fact that $b>\frac{12}{7}$, and H\"{o}lder's inequality that
\begin{eqnarray}\label{b110}
I_{8}^{1}&\lesssim&\left(\int_{0}^{t}\int_{\mathbb{R}}\left(1+\theta^{2}\right)u_{x}^{2}\right)^{\frac{1}{2}}
\left(\int_{0}^{t}\int_{\mathbb{R}}u_{x}^{4}\right)^{\frac{1}{2}}\lesssim
   1+\left\|\theta\right\|^{\frac{3\ell+6}{4}}_{\infty}\nonumber\\
   &\lesssim& 1+Y(t)^{\frac{3\ell+6}{4(2b+6)}}\leq \epsilon Y(t)+C\left(\epsilon\right),
\end{eqnarray}
and
\begin{eqnarray}\label{b111}
I_{8}^{2}&&\leq\epsilon X(t) +C\left(\epsilon\right)\left(1+\|\theta\|_{\infty}^{(b-3)_{+}}\right)\int_{0}^{t}\int_{\mathbb{R}}u_{x}^{4}\leq\epsilon X(t)
   +C\left(\epsilon\right)\left(1+\left\|\theta\right\|^{\frac{3\ell+2(b-3)_{+}}{2}}_{\infty}\right)\nonumber\\
   &&\leq\epsilon X(t)+C\left(\epsilon\right)\left(1+Y(t)^{\frac{3\ell+2(b-3)_{+}}{2(2b+6)}}\right)
   \leq \left(X(t)+Y(t)\right)+C\left(\epsilon\right).
\end{eqnarray}
Then inserting \eqref{b110}, \eqref{b111} into \eqref{b108}, we obtain
\begin{eqnarray}\label{b112}
I_{8}\leq \left(X(t)+Y(t)\right)+C\left(\epsilon\right).
\end{eqnarray}

For the term $I_9$, \eqref{b66} tells us that
\begin{eqnarray}\label{b113}
I_{9}&\leq&\underbrace{\int_{0}^{t}\int_{\mathbb{R}}\left(1+\theta^{b+1}\right)\left|\theta_{x}v_{x}u_{x}\right|}_{I^1_9}
+\underbrace{\int_{0}^{t}\int_{\mathbb{R}}\frac{\kappa(v,\theta)\left|\theta\theta_{x}u_{xx}\right|}{v}}_{I_9^2}.
\end{eqnarray}
We have from \eqref{b3}, \eqref{bz57}, \eqref{b50}, the fact $b>0$, Sobolev's and H\"{o}lder's inequality that
\begin{eqnarray}\label{b114}
I_9^1&&\lesssim\int_{0}^{t}\int_{\mathbb{R}}\frac{\kappa(v,\theta)\theta^{2}_{x}}{v\theta^{2}}\cdot\left(1+\theta^{b+4}\right)
+\int_{0}^{t}\left\|u_{x}\right\|\left\|u_{xx}\right\|\|v_{x}\|^{2}d\tau\nonumber\\
&&\lesssim 1+\left\|\theta\right\|^{b+4}_{\infty}+\left(1+\left\|\theta\right\|_{\infty}\right)
\left(\int_{0}^{t}\int_{\mathbb{R}}u_{x}^{2}\right)^{\frac{1}{2}}
\left(\int_{0}^{t}\int_{\mathbb{R}}u_{xx}^{2}\right)^{\frac{1}{2}}
\lesssim 1+\left\|\theta\right\|^{b+4}_{\infty}+\left\|\theta\right\|^{\frac{\ell+3}{2}}_{\infty}\nonumber\\
&&\lesssim 1+Y(t)^{\frac{b+4}{2b+6}}+Y(t)^{\frac{\ell+3}{2(2b+6)}}\leq \epsilon Y(t)+C\left(\epsilon\right),
\end{eqnarray}
and
\begin{eqnarray}\label{b115}
I_9^2
&&\lesssim \left(\int_{0}^{t}\int_{\mathbb{R}}u_{xx}^{2}\right)^{\frac{1}{2}}
\left(\int_{0}^{t}\int_{\mathbb{R}}\frac{\kappa(v,\theta)\theta^{2}_{x}}{v\theta^{2}}\cdot\left(1+\theta^{b+4}\right)\right)^{\frac{1}{2}}
\lesssim 1+\left\|\theta\right\|^{\frac{b+\ell+4}{2}}_{\infty}\nonumber\\
&&\lesssim 1+Y(t)^{\frac{b+\ell+4}{2(2b+6)}}\leq \epsilon Y(t)+C\left(\epsilon\right).
\end{eqnarray}
The combination of \eqref{b113}-\eqref{b115} shows that
\begin{eqnarray}\label{b116a}
I_9\leq \epsilon Y(t)+C\left(\epsilon\right).
\end{eqnarray}

As for the term $I_{10}$, in light of \eqref{b3}, \eqref{bz57}, \eqref{b58}, \eqref{b66}, and Sobolev's inequality, one has
\begin{eqnarray}\label{b116}
    I_{10}&\leq&\epsilon X(t)+C(\epsilon)\int_{0}^{t}\left\|\frac{\kappa(v,\theta)\theta_{x}}{v}\right\|^{2}_{L^{\infty}}\|v_{x}\|^{2}d\tau\nonumber\\
   &\leq&\epsilon X(t)+C(\epsilon)\left(1+Y^{\frac{1}{2b+6}}\right)
   \int_{0}^{t}\int_{\mathbb{R}}\left|\frac{\kappa(v,\theta)\theta_{x}}{v}\right|
   \left|\left(\frac{\kappa(v,\theta)\theta_{x}}{v}\right)_{x}\right|\nonumber\\
   &\leq&\epsilon X(t)+C(\epsilon)\left(1+Y^{\frac{1}{2b+6}}\right)
   \left(\int_{0}^{t}\int_{\mathbb{R}}\theta^{2}\kappa\left|
   \left(\frac{\kappa(v,\theta)\theta_{x}}{v}\right)_{x}\right|^{2}\right)^{\frac{1}{2}}
\left(\int^{t}_{0}\int_{\mathbb{R}}\frac{\kappa(v,\theta)\theta^{2}_{x}}{v\theta^{2}}\right)^{\frac{1}{2}}\nonumber\\
 &\leq&\epsilon X(t)
    +C(\epsilon)\left(1+Y(t)^{\frac{1}{2b+6}}\right)
    \underbrace{\left(\int_{0}^{t}\int_{\mathbb{R}}\left(1+\theta^{b+2}\right)
    \left|\left(\frac{\kappa(v,\theta)\theta_{x}}{v}\right)_{x}
    \right|^{2}\right)^{\frac{1}{2}}}_{J^{\frac{1}{2}}}.\nonumber\\
\end{eqnarray}
Moreover, we can conclude from $\eqref{a1}_{3}$, \eqref{a3}, \eqref{b2}, \eqref{b3}, and \eqref{b50} that

\begin{eqnarray}\label{b117}
J&&\lesssim\int_{0}^{t}\int_{\mathbb{R}}\left(1+\theta\right)^{b+2}
\left(e_{\theta}^{2}(v,\theta)\theta_{t}^{2}+\theta^{2}p_{\theta}^{2}(v,\theta)u_{x}^{2} +u_{x}^{4}+\phi^{2}z^{2}\right)\nonumber\\
&&\lesssim 1+\left\|\theta\right\|^{b+2+\frac{3\ell}{2}}_{\infty}+\left\|\theta\right\|^{b+2+\beta}_{\infty}
+\int_{0}^{t}\int_{\mathbb{R}}\left(\left(1+\theta\right)^{b+8}\theta_{t}^{2}+\left(1+\theta\right)^{b+10}u_{x}^{2}\right)\nonumber\\
&&\lesssim 1+\left\|\theta\right\|^{b+2+\frac{3\ell}{2}}_{\infty}+\left\|\theta\right\|^{b+2+\beta}_{\infty}
+\left\|\theta\right\|^{b+11}_{\infty}+\left(1+\left\|\theta\right\|^{5}_{\infty}\right)X(t)\nonumber\\
&&\lesssim 1+X(t)+X(t)Y(t)^{\frac{5}{2b+6}}+Y(t)^{\frac{b+11}{2b+6}}+Y(t)^{\frac{3\ell+2(b+2)}{2(2b+6)}}+Y(t)^{\frac{b+2+\beta}{2b+6}}.
\end{eqnarray}
Plugging \eqref{b117} into \eqref{b116} then employing Young's inequality, the fact $b>\frac{8}{9}$ and $0\leq\beta<3b+8$, we can obtain

\begin{eqnarray}\label{b130}
I_{10}&&\leq\epsilon\left(X(t)+Y(t)\right)+C\left(\epsilon\right).
\end{eqnarray}

It suffices to bound the term $I_{11}$ and $I_{12}$. For this purpose, making use of \eqref{b2}, \eqref{b3}, \eqref{b58}, \eqref{b66}, and
the assumption $0\leq\beta<b+9$, we get

\begin{eqnarray}\label{b135}
&&I_{11}+I_{12}\nonumber\\
&\leq&C\left(1+\|\theta\|_{\infty}^{\frac{\beta+3}{2}}\right)\left(\int_{0}^{t}\int_{\mathbb{R}}\phi z^{2}\right)^{\frac{1}{2}}\left(\int_{0}^{t}\int_{\mathbb{R}}\frac{u^{2}_{x}}{\theta}\right)^{\frac{1}{2}}+ \epsilon X(t)+C\left(\epsilon\right)\int_{0}^{t}\int_{\mathbb{R}}\left(1+\theta^{b+\beta-3}\right)\phi z^{2} \nonumber\\
&\leq& \epsilon X(t)+C\left(\epsilon\right)\left(1+Y(t)^{\frac{\beta+3}{4b+12}}+Y(t)^{\frac{b+\beta-3}{2b+6}}\right)\nonumber\\
&\leq&\epsilon\left(X(t)+Y(t)\right)+C\left(\epsilon\right).
\end{eqnarray}

Combining all the above estimates and by choosing $\epsilon>0$ small enough, we can deduce the estimate \eqref{b62} immediately. This completes the proof of our lemma.
\end{proof}

The uniform upper bound of $\theta\left(t,x\right)$ follows directly from Lemma 3.1 and Lemma 3.2. In fact, we can obtain the following lemma by utilizing the argument developed in \cite{Liao-Zhao-JDE-2018}.

\begin{lemma} Under the assumptions listed in Lemma 2.1, there exist positive constants $\overline{\Theta}$ and $C$ which depend only on the initial data $(v_0(x), u_0(x), \theta_0(x), z_0(x))$, such that
  \begin{equation}\label{b151}
 \theta\left(t,x\right) \leq\overline{\Theta},\quad \forall
 \left(t,x\right)\in[0,T] \times\mathbb{R}.
  \end{equation}
and
\begin{equation}\label{b152}
 \left\|\left(v-1, u, \theta-1, z\right)\left(t\right)\right\|_{1}^{2}+\int_{0}^{t}\left\|\left(\sqrt{\theta}v_{x}, \theta_{t} \right)(s)\right\|^{2}+\left\|\left(u_{x}, \theta_{x}, z_{x}\right)\left(s\right)\right\|_{1}^{2}ds\leq C
\end{equation}
Moreover, for each $0\leq s\leq t\leq T$ and $x\in\mathbb{R}$, we have
 \begin{equation}\label{b180}
  \theta\left(t,x\right)\geq \frac{C\min\limits_{x\in\mathbb{R}}\{\theta(s,x)\}}{1+(t-s)\min\limits_{x\in\mathbb{R}}\{\theta(s,x)\}}.
  \end{equation}

\end{lemma}

With Lemma 2.1-Lemma 3.3 in hand, we can complete the proof of Theorem 1.1 by using the standard argument developed in \cite{Liao-Zhao-JDE-2018}, and we omit the details for brevity.

\begin{center}
{\bf Acknowledgement}
\end{center}
The research of Yongkai Liao is supported by National Postdoctoral Program for Innovative Talents of China No. BX20180054. The author express much gratitude to Professor Huijiang Zhao for his support and advice.

\end{document}